\documentclass[12pt]{amsart}

\usepackage{pdfsync}
\usepackage{xcolor}

\usepackage{geometry}
\usepackage{lscape}
\usepackage{graphicx}
\usepackage{epstopdf}    
\numberwithin{equation}{section}

\theoremstyle{plain}  
\newtheorem{thm}[equation]{Theorem}
\newtheorem{prop}[equation]{Proposition}

\newtheorem{lemma}[equation]{Lemma}

\theoremstyle{definition}  

\newtheorem{remark}[equation]{Remark}

\usepackage{amsmath, amscd, amssymb, fullpage, setspace, , verbatim, lscape,hyperref, xypic}
\usepackage{amsthm}
\usepackage[shortlabels]{enumitem}
\usepackage[noadjust]{cite}

\usepackage{color}
\usepackage{lipsum}
\newcommand{\ind}{\hspace{15 pt}}

\newcommand{\Tor}{\text{Tor}}
\newcommand{\ZZ}{\mathbb{Z}}

\newcommand{\ER}{\mathbb{ER}}
\renewcommand{\MR}{\mathbb{MR}}

\newcommand{\cp}{\mathbb{CP}^\infty}
\newcommand{\vhat}{\widehat{v}}
\newcommand{\uhat}{\widehat{u}}
\newcommand{\chat}{\widehat{c}}
\newcommand{\ubar}{\uhat^\ast}
\newcommand{\uub}{\uhat \ubar}
\newcommand{\phat}{\widehat{p}}
\newcommand{\N}{\mathcal{N}}
\newcommand{\Nr}{N_*^{\text{res}}}
\newcommand{\imN}{\mathrm{im}(\N_*)}
\newcommand{\imNr}{\mathrm{im}(\N_*^{\text{res}})}
\newcommand{\imn}{\mathrm{im}(N_*)}
\newcommand{\imnr}{\mathrm{im}(N_*^{\text{res}})}
\newcommand{\Nt}{\widetilde{N}}
\newcommand{\inv}{\text{inv}}

\begin{document}
\pagestyle{plain}

\title
{The Real Johnson-Wilson cohomology of $\cp$}
\author{Vitaly Lorman}
\address{Department of Mathematics, Johns Hopkins University, Baltimore, USA}
\email{vlorman@math.jhu.edu}
\date{\today}


{\abstract
We completely compute the Real Johnson-Wilson cohomology of $\cp$. Applying techniques from equivariant stable homotopy theory to the Bockstein spectral sequence, we produce permanent cycles and solve extension problems to give an explicit description of the ring $ER(n)^*(\cp)$.}
\maketitle

\section{Introduction}\label{sec:1}
The complex cobordism spectrum, $MU$, carries an action of $C_2$, the group of order two, coming from complex conjugation and may be constructed as a genuine (indexed on a complete universe) $C_2$-equivariant spectrum, $\MR$ \cite{Lan68, Fuj, AM78}. At the prime 2, the Johnson-Wilson spectrum, $E(n)$, is an $MU$-algebra with coefficients
$$E(n)^*=\ZZ_{(2)}[v_1, \dots, v_{n-1}, v_n^{\pm 1}]$$
where $v_k$ is in cohomological degree $-2(2^k-1)$. $E(n)$ may also be constructed as a genuine $C_2$-equivariant spectrum, Real Johnson-Wilson Theory, $\ER(n)$ \cite{HK01}.  On the category of $C_2$-equivariant spaces, $\ER(n)$ naturally gives rise to a multiplicative cohomology theory valued in commutative $\MR$-algebras \cite{KLW16c}. Its underlying nonequivariant spectrum is $E(n)$. Let $ER(n)$ denote the fixed point spectrum $\ER(n)^{C_2}$. 

Fixed point spectra associated to $\MR$ have proved to be powerful tools for algebraic topologists, for example in Hill, Hopkins, and Ravenel's resolution of the Kervaire invariant one problem \cite{HHR09} (in all dimensions besides 126). At heights $n=1$ and $2$, the $ER(n)$ are familiar cohomology theories.} $\ER(1)$ is $\mathbb{KR}_{(2)}$, Atiyah's Real K-theory with underlying nonequivariant spectrum $KU_{(2)}$, and $ER(1)$ is $KO_{(2)}$, real K-theory. After a suitable completion, the spectrum $ER(2)$ is (additively) equivalent to the spectrum $TMF_0(3)$ of topological modular forms with level structure \cite{MR09, HM15}, and so the results of this paper apply there. At present, this identification provides the most powerful approach to computing $TMF_0(3)$-cohomology (for example, see \cite{LO16}).

Just as there is a fibration
$$\xymatrix{\Sigma KO \ar@{->}[r]^-\eta &KO \ar@{->}[r]&KU}, \ind \eta \in KO^{-1}$$
there is for each $n$ a fibration
$$\xymatrix{\Sigma^{\lambda(n)} ER(n) \ar@{->}[r]^-{x(n)}&ER(n) \ar@{->}[r]&E(n)}$$
where $x(n) \in ER(n)^{-\lambda(n)}$ is $(2^{n+1}-1)$-nilpotent and $\lambda(n)=2(2^n-1)^2-1=2^{2n+1}-2^{n+2}+1$. This was constructed in ~\cite{KW07} and leads to a Bockstein spectral sequence of the form
$$E_1^{i, j}=E(n)^{i\lambda(n)+j-i}(X) \Rightarrow ER(n)^{j-i}(X).$$

Our main result is a computation of $ER(n)^*(\cp)$ for all $n$. $ER(n)$ is not a complex oriented theory, and so the computation is nontrivial. The Atiyah-Hirzebruch spectral sequence is unwieldy, so we use the above Bockstein spectral sequence instead. Even this has nontrivial higher differentials, but what makes $ER(n)^*(\cp)$ computable is the fact that, after some rearranging, the only interesting differential is $d_1$ and the higher differentials all play out in the coefficients. We develop methods of producing permanent cycles and obtain control over the extension problems in the Bockstein spectral sequence to produce a complete description of the ring $ER(n)^*(\cp)$.

In the case of $ER(1)=KO_{(2)}$, this problem has a long history. $KO^*(\cp)$ was first computed in degree zero by Sanderson \cite{San64}, then in all degrees with ring structure on $KO^{\text{even}}(\cp)$ by Fujii \cite{Fuj67}. Yamaguchi \cite{Yam07} gave the first complete description of $KO^*(\cp)$ as a ring. All three computations use the Atiyah-Hirzebruch spectral sequence. Bruner and Greenlees \cite{BG10} computed the connective real $K$-theory of $\cp$ using the Bockstein spectral sequence. Our result gives the $ER(n)$-cohomology of $\cp$ for all $n$ in the same level of detail as Yamaguchi's description and reproduces $n=1$ as a nearly degenerate case (see Remarks \ref{remark:KOintro} and \ref{remark:KO}). The answer for $n>1$ is richer in the sense that it yields a great deal more 2-torsion.

This paper forms part of a program to compute the $ER(n)$-cohomology of basic spaces. The spaces whose $ER(n)$-cohomology is known at present may be divided into two families. Building on the computations in \cite{KW08a, KW08b, KW14}, the results of \cite{KLW16a} identify a class of spaces whose $ER(n)$-cohomology is directly computed from $E(n)$-cohomology (which is known) by base change. These include $BO(q)$ for $q \leq \infty$, the connective covers $BSO$, $BSpin$ and $BString$ (the last for $n\leq 2$ only), and half of all Eilenberg MacLane spaces---those of the form $K(\ZZ, 2k+1)$ and $K(\ZZ/2^q, 2k)$. The space of interest in this paper, $\cp=K(\ZZ, 2)$, does not belong to this family and is the first space of `exotic type' whose $ER(n)$-cohomology is nevertheless computable. Our computations here open the door to the results of \cite{KLW16b}, which deal with the $ER(2)$-cohomology of $B\ZZ/2^q$ and truncated complex projective spaces $\mathbb{CP}^k$. The computation of $ER(n)^*(\cp)$ also points the way toward the $ER(n)$-cohomology of the spaces $\prod_{i=1}^j \cp$, $BU(q)$ and its connective covers, and the other half of the Eilenberg MacLane spaces.

One motivation for developing $ER(n)$ (and $ER(2)$ in particular) as a computable theory is the applications to proving nonimmersions of projective spaces. Kitchloo and Wilson \cite{KW08a, KW08b} and Banerjee \cite{Ban13} established new nonimmersion results for real projective spaces by constructing obstructions in $ER(2)$-cohomology. These obstructions are given by powers of a generating class surviving beyond the skeletal truncation of the projective space---they are 2-torsion and undetectable by any complex oriented theory. The computations in \cite{KLW16b} reveal the same sorts of extra powers of a generating class (the class $\phat_1$ described in the next paragraph) in the $ER(2)$-cohomology of $\mathbb{CP}^k$, and one goal of this computational program is to attack the nonimmersion problem and related questions for complex projective spaces.

Before stating the main result, we fix some notation. For any $z \in E(n)^{2k}(\cp)$, let $\widehat{z}$ denote $z v_n^{k(2^n-1)}$. In particular, we have $\vhat_i:=v_iv_n^{-(2^i-1)(2^n-1)}$ and, letting $u\in E(n)^2(\cp)$ denote the complex orientation, $\uhat:=v_n^{2^n-1}u$. Let $c$ denote the involution on $E(n)^*(\cp)$ coming from the action of $C_2$ on $\ER(n)$. We show that $c(\uhat)=\uhat^*$ is given by the power series $v_n^{2^n-1}[-1]_{F}(u)$, where $F$ is the formal group law over $E(n)^*$ with $[2]_F(u)=v_0u+_F v_1u^{2}+_F \dots +_Fv_nu^{2^n}$. Under the identification $\cp=BSO(2)$, the product $\uub \in E(n)^*(\cp)$ is $v_n^{2(2^n-1)}$ times the first Pontryagin class. In Proposition \ref{thm:uubar}, we describe a class in $ER(n)^*(\cp)$ which lifts $\uub \in E(n)^*(\cp)$. We denote the lift by $\phat_1$. In $E(n)^*(\cp)$, we also have the sum $\uhat+\ubar$. We show that this lifts to $ER(n)^*(\cp)$ as a power series in $\phat_1$, denoted $\xi(\phat_1)$. We have the main theorem of this paper.

\begin{thm}\label{thm:SES} There is a short exact sequence of modules over $ER(n)^*$
$$0 \longrightarrow \imnr \longrightarrow ER(n)^*(\cp) \longrightarrow \frac{ER(n)^*[[\phat_1]]}{(\xi(\phat_1))}\longrightarrow 0$$
where $\imnr$ is the image of the restricted norm
$$\xymatrix{\ZZ_{(2)}[\vhat_1, \dots, \vhat_{n-1}, v_n^{\pm 2}][[\uub]]\{\uhat, v_n\uhat\} \subset E(n)^*(\cp) \ar@{->}[r]^-{N_*}  & ER(n)^*(\cp)}$$
such that for all $z$, $N_*(z)$ maps to $z+c(z)$ under the map ${ER(n)^*(\cp) \longrightarrow E(n)^*(\cp)}$.
\end{thm}
\begin{remark} Although the middle and right terms of the short exact sequence of Theorem \ref{thm:SES} are rings, the $ER(n)^*$-module $\imnr$ is not an ideal of $ER(n)^*(\cp)$ and the right hand map is not a ring homomorphism. However, we give a complete answer for $ER(n)^*(\cp)$ as an algebra in terms of generators and relations in Theorem \ref{thm:final}. A simpler answer than either Theorem \ref{thm:SES} or \ref{thm:final} is given by restricting to degrees multiples of $2^{n+2}$ (note that $ER(n)$ is $2^{n+2}(2^n-1)$-periodic), though this portion of it contains none of the 2-torsion:
$$ER(n)^{2^{n+2}\ast}(\cp)=ER(n)^{2^{n+2}*}(pt)[[\phat_1]]$$
\end{remark}

\begin{remark} Note that neither the complex orientation $u$ nor $\uhat$ lift to $ER(n)^*(\cp)$, but $\uub$ does. The characteristic class $\xi(\phat_1)$ by which we quotient in the right hand term above is not zero in $ER(n)^*(\cp)$ but is in the image of the norm, $\xi(\phat_1)=N_*(\uhat)$. It has geometric significance as we discuss further in Remark \ref{remark:uhat+ubar}.\end{remark}

\begin{remark}\label{remark:KOintro} In Theorem \ref{thm:SES} for $n=1$, it turns out that $\xi(\phat_1)=-\phat_1$, and so the right hand term reduces to the coefficients $ER(1)^*$. This is not true for $n>1$ (see Remark \ref{remark:KO}). \par 

\end{remark}

In sections \ref{sec:2} and \ref{sec:3} we review the Bockstein spectral sequence and describe the computation for the coefficients. In sections \ref{sec:4} and \ref{sec:5} we begin the computation for $\cp$, identifying the key permanent cycle, $\uub$, and giving a convenient reformulation of $E_1^{*, *}$. In section \ref{sec:6} we compute $E_2^{*, *}$. From there, sections \ref{sec:7}-\ref{sec:9} break up the Bockstein spectral sequence into a short exact sequence of spectral sequences and show that the remainder of the computation happens in the coefficients via a Landweber flatness argument. In section \ref{sec:10}, we prove Theorem \ref{thm:SES}, describe the multiplicative structure, and state the most explicit form of the answer as an algebra over $ER(n)^*$. Finally, section \ref{sec:11} describes the very clean form of the answer that occurs after a certain completion. The appendix at the end contains some key equivariant lemmas necessary for our computations.

This paper forms part of the author's thesis \cite{Lor16}.

\textbf{Acknowledgements:} This work would not be possible without the patience, support, and enthusiasm of my advisor Nitu Kitchloo. I am also deeply grateful to Steve Wilson for his interest in this project and insightful comments. Finally, I am grateful for the comments and suggestions of an anonymous referee.

\section{$\ER(n)$ and the Bockstein spectral sequence} \label{sec:2}
We begin by reviewing some facts about $\ER(n)$ we need to set up our computational machinery. Let $\alpha$ denote the sign representation of $C_2$. By a genuine $C_2$-equivariant spectrum $\mathbb{E}$, we mean a collection of spaces $\mathbb{E}_V$ ranging over finite-dimensional $C_2$-representations $V=s+t\alpha$ together with a transitive system of based $C_2$-equivariant homeomorphisms
$$\mathbb{E}_V \longrightarrow \Omega^{W-V}\mathbb{E}_W, \ind \text{for }V \subseteq W.$$
Such spectra represent bigraded cohomology theories $\mathbb{E}^\star(-)$ given by 
$$\mathbb{E}^{s+t\alpha}(-)=[-, \Sigma^{s+t\alpha}\mathbb{E}]^{C_2}.$$
The $C_2$-action means there is an involution on $\mathbb{E}$. Letting $\iota^*\mathbb{E}$ denote the underlying nonequivariant spectrum, there is an induced involution on the (nonequivariant) cohomology groups $(\iota^*\mathbb{E})^*(-)$. This is the same $c$ described in the introduction.

The $\ER(n)$ are genuine $C_2$-equivariant spectra. We draw attention to two classes in the coefficients, $\ER(n)^\star(pt)=\ER(n)^\star$. As shown in ~\cite{KW07}, there is an invertible class $y(n) \in \ER(n)^{-\lambda(n)-\alpha}$. It restricts to the (nonequivariant) class $v_n^{2^n-1} \in E(n)^{-\lambda(n)-1}$. Additionally, there is a class $x(n) \in \ER(n)^{-\lambda(n)}=ER(n)^{-\lambda(n)}$ with $x(n)^{2^{n+1}-1}=0$. Henceforth, we drop the `$n$' and simply write $x, y,$ and $\lambda$.

In \cite{KW07}, Kitchloo and Wilson construct the fibration
$$\begin{CD}\Sigma^{\lambda}ER(n) @>x>> ER(n) @>>> E(n)\end{CD}.$$
Applying  $[X, -]$ yields an exact couple
$$\xymatrix{
ER(n)^*(X) \ar@{->}[rr]^-{x}
&&ER(n)^*(X) \ar@{->}[ld] \\
&E(n)^*(X) \ar@{->}[lu]}$$

and produces the Bockstein spectral sequence (BSS). 

\begin{remark} Depending on whether one truncates the multiplication-by-$x$ tower, there are two spectral sequences that can arise from the above. One converges to $ER(n)^*(X)$ (as in \cite{KW14}), the other to $0$ (as in \cite{KW08a}, \cite{KW08b}, and \cite{KLW16b}). In the latter case, one must go back to reconstruct the answer from the differentials. Both have their advantages and ultimately contain equivalent information, but it is the truncated BSS converging to $ER(n)^*(X)$ that we use in this paper.
\end{remark}

The BSS has the following properties.

\begin{thm} \label{thm:bss}\cite{KW14}
\leavevmode

\begin{enumerate}[(i)]
\item There is a first and fourth quadrant spectral sequence of $ER(n)^*$-modules, $E_r^{i, j} \Rightarrow ER(n)^{j-i}(X)$. The differential $d_r$ has bidegree $(r, r+1)$ for $r \geq 1$.
\item The $E_1$-term is given by
$$E_1^{i, j}=E(n)^{i\lambda +j-i}(X)$$
with
$$d_1(z)=v_n^{1-2^n}(1-c)(z)$$
where $c(v_i)=-v_i$. The differential $d_r$ increases cohomological degree by $1+r\lambda$ between the appropriate subquotients of $E(n)^*(X)$.
\item $E_{2^{n+1}}(X)=E_\infty(X)$, which is described as follows. Filter $M=ER(n)^*(X)$ by $M_r=x^rM$ so that
$$M=M_0 \supset M_1 \supset M_2 \supset \dots \supset M_{2^{n+1}-1}=\{0\}.$$
Then $E_\infty^{r, *}(X)$ is canonically isomorphic to $M_r/M_{r+1}$.
\item $d_r(ab)=d_r(a)b+c(a)d_r(b)$. In particular, if $c(z)=z \in E_r(X)$ then $d_r(z^2)=0$, $r>1$.
\end{enumerate}
\end{thm}

\begin{remark} As in \cite{KW14}, we note that when $X$ is a space, there is a canonical class in $E_1^{1, 1-\lambda}$ that corresponds to $1 \in E(n)^{\lambda+1-\lambda-1}(X)=E(n)^0(X)$ and is a permanent cycle representing $x\in ER(n)^{-\lambda}$. We abuse notation and give its representative in $E_1^{1, 1-\lambda}$ the name $x$ as well. Note that though $x$ is a permanent cycle, $x^{2^{n+1}-1}$ does not survive the spectral sequence and is equal to zero in $ER(n)^*(X)$.  We may rewrite the $E_1$-page to index the vertical lines by powers of $x$:
$$E_1^{*, *}=E_1^{0, *}[x]=E(n)^*(X)[x]$$
$$d_1(z)=v_n^{1-2^n}(1-c)(z)x, \ind v_n \in E_1^{0, 2(1-2^n)}$$
\end{remark}
\begin{remark} To make things even more confusing, the representative of $x=x(n)$ in $E_1^{1, 1-\lambda}$ was previously called $y$ in \cite{KW14} and is not the same as $y(n)\in  \ER(n)^{-\lambda(n)-\alpha}$ as described above. Since our $x \in E_1^{1, 1-\lambda}$ represents $x=x(n) \in ER(n)^{-\lambda}$, we choose the lesser of two evils and henceforth use our notation instead.
\end{remark}


\section{The spectral sequence for $X=pt$} \label{sec:3}
When $X=pt$, we have
$$E_1^{*, *}=E(n)^*=\ZZ_{(2)}[v_1, \dots, v_{n-1}, v_n^{\pm 1}][x], \ind |v_k|=-2(2^k-1).$$
None of the generators $v_k$ are permanent cycles as $c(v_k)=-v_k$. However, there is a trick we can do to replace $v_k$ for $k<n$ by permanent cycles $\vhat_k$. As in~\cite{HK01}, each class $v_k \in MU^{-2(2^k-1)}$ has an equivariant lift in $\MR^{-(2^k-1)(1+\alpha)}$. For $0 \leq k \leq n$, the $\MR$-algebra structure on $\ER(n)$ produces these classes in $\ER(n)^\star$. We may use $y \in \ER (n)^{-\alpha-\lambda}$ to shift the ``diagonal'' $v_k$ classes to integer grading. For $0 \leq k < n$, let $\vhat_k \in \ER (n)^{(2^k-1)(\lambda-1)}=ER(n)^{(2^k-1)(\lambda-1)}$ denote $v_ky^{-(2^k-1)}$. Then by construction, this class restricts to $v_kv_n^{-(2^k-1)(2^n-1)} \in E(n)^{(2^k-1)(\lambda-1)}$ and represents a permanent cycle in $E_1^{0, (2^k-1)(\lambda-1)}$. 

We have now shifted all of the differentials onto powers of $v_n$. Let $R_n=\ZZ_{(2)}[\vhat_1, \dots, \vhat_{n-1}]$, $I_j=(2, \vhat_1, \dots, \vhat_{j-1})$, and $I_0=(0)$. The Bockstein spectral sequence computing $ER(n)^*$ goes as follows.

\begin{thm} \cite{KW14} In the spectral sequence $E_r(\text{pt}) \Rightarrow ER(n)^*$,

\begin{enumerate}[(i)]
\item 
$$E_1^{*, *} \cong \ZZ_{(2)}[\vhat_1, \vhat_2, \dots, \vhat_{n-1}, v_n^{\pm 1}][x]$$
That is,
$$E_1^{m, *}=\ZZ_{(2)}[\vhat_1, \dots, \vhat_{n-1}, v_n^{\pm 1}] \text{  on  } x^m.$$
\item The only non-zero differentials are generated by
$$d_{2^{k+1}-1}(v_n^{-2^k})=\vhat_kv_n^{-2^{n+k}}x^{2^{k+1}-1} \text{  for  } 0 \leq k \leq n.$$
\item $E_{2^k}^{*, *}=E_{2^k+1}^{*, *}= \dots =E_{2^{k+1}-1}^{*, *}$, for $0 \leq k \leq n$, and $E_{2^{n+1}}^{*, *}=E_\infty^{*, *}$.
\item For $0 \leq j < k \leq n+1$,
$$E_{2^k}^{m, *}=R_n[v_n^{\pm 2^k}]/I_j \bigoplus_{j<i<k} I_iR_n[v_n^{\pm 2^{i+1}}]v_n^{2^i}/I_j \text{  on  } x^m$$
when $2^j-1\leq m < 2^{j+1}-1$.
\item For $0 < k \leq n+1$ and $2^k-1 \leq m$,
$$E_{2^k}^{m, *}=R_n[v_n^{\pm 2^k}]/I_k \text{  on  } x^m.$$
\end{enumerate}

\end{thm}
\begin{remark} Note that $v_n$ does not survive the spectral sequence, but $v_n^{2^{n+1}}$ does. Since it is invertible, this makes $ER(n)$ periodic with period $|v_n^{-2^{n+1}}|=2^{n+2}(2^n-1)$. \end{remark}


\section{The spectral sequence for $X=\cp$} \label{sec:4}
The BSS for $\cp$ starts with
$$E_1^{\ast, \ast}=E(n)^*(\cp)[x]=E(n)^*[[u]][x].$$
Again, we hat off $v_i$, $0 \leq i <n$ so that $E(n)^*=\ZZ_{(2)}[\vhat_1, \dots, \vhat_{n-1}, v_n^{\pm 1}]$. As described in the introduction, in general, for a class $z \in E(n)^{2j}(X)$, we set 
$$\widehat{z}:=v_n^{j(2^n-1)}z\in E(n)^{j(1-\lambda)}(X).$$
However, note that for arbitrary $z$, $\widehat{z}$ need not be a permanent cycle. In fact, $\uhat=v_n^{2^n-1}u \in E(n)^{1-\lambda}(\cp)$ is not. In any case, we replace $u$ by $\uhat$ as the power series generator of $E(n)^*(\cp)$, which is valid since $v_n$ is a unit. We may similarly hat off the coefficients of the formal group law so that $\widehat{F}(\uhat_1, \uhat_2)$ is a homogenous expression of degree $1-\lambda$ and satisfies $v_n^{2^n-1}F(u_1, u_2)=\widehat{F}(\uhat_1, \uhat_2)$.

We now have
$$E_1^{\ast, \ast}=\ZZ_{(2)}[\vhat_1, \dots, \vhat_{n-1}, v_n^{\pm 1}][[\uhat]][x]$$
with the following bidegrees:
$$|\widehat{v_k}|=\left(0, \frac{1-\lambda}{2}|v_k|\right)=(0, (\lambda-1)(2^k-1)), \ind |v_n|=(0, -2(2^n-1))$$
$$|\widehat{u}|=(0, 1-\lambda), \ind |x|=(1, 1-\lambda)$$

To compute $d_1$, by Theorem \ref{thm:bss}(ii), we need the action of $c$ on $E_1^{*, *}$. The classes $\vhat_k$, $0 \leq k <n$, as well as $x$ are permanent cycles and in particular have trivial $c$-action. We have $c(v_n)=-v_n$. It remains to identify $c(\uhat)$.

\begin{lemma}\label{lemma:cuhat}
$$c(\widehat{u})=[-1]_{\widehat{F}}(\widehat{u})$$
\end{lemma}
\begin{proof} We view $u$ as an equivariant map $u: \cp \longrightarrow \Sigma^{1+\alpha}E(n)$. The diagram
$$\xymatrix{
\cp \ar@{->}^-{\inv}[d] \ar@{->}^-u[r] & S^{1+\alpha} \wedge E(n) \ar@{->}^-{(-1) \wedge c}[d]\\
\cp \ar@{->}^-u[r] & S^{1+\alpha} \wedge E(n)
}$$
commutes, where $\inv$ denotes the involution on $\cp$ classifying the conjugate line bundle with $\inv^*(u)=[-1]_F(u)$. The above diagram shows that $c(u)=-[-1]_F(u)$. Then on $\uhat=v_n^{2^n-1}u$, we have
$$c(\uhat)=c(v_n^{2^n-1}u)=-v_n^{2^n-1}c(u)=v_n^{2^n-1}[-1]_F(u)=[-1]_{\widehat{F}}(\uhat).$$
\end{proof}

From now on, we let $\ubar$  denote $c(\uhat)=[-1]_{\widehat{F}}(\uhat)$. For future reference, it will be helpful to have some terms of this power series, so we pause to derive some formulas. 

\begin{lemma}\label{lemma:ubarcong} We have the following congruences in $E(n)^*(\cp)$: 
\begin{align*}
\ubar &\equiv -\uhat \text{ mod }(\uhat^2)\\
\ubar &\equiv \uhat + \vhat_k \uhat^{2^k} \text{ mod }(\vhat_0, \dots, \vhat_{k-1}, \uhat^{2^k+1}) \ind \text{ for }0 < k < n
\end{align*}
\end{lemma}
\begin{proof} Both follow from the formula for the $2$-series
$$[2]_{\widehat{F}}(\uhat)={\sum_{i=0}^n}_{\widehat{F}}\vhat_i\uhat^{2^i}$$
and the equation 
$$\ubar+_{\widehat{F}}[2]_{\widehat{F}}(\uhat)=\uhat.$$
\end{proof}


\section{A topological basis for $E_1^{*,*}$} \label{sec:5}
The next step in computing $d_1$ is finding a convenient topological basis for $E_1^{*, *}$. To that end, we identify a large collection of permanent cycles in our spectral sequence. Note that $E_1^{*, *}$ is a power series ring over $E(n)^*[x]$, and throughout, by a basis for $E_1^{*, *}$ we mean a topological basis.

\begin{prop} \label{thm:uubar} $\uhat \ubar$ is a permanent cycle\end{prop}
\begin{proof}
Our starting point is $\mathbb{ER}(n)^\star(BU(2))$. It may be computed using the Real Atiyah-Hirzebruch spectral sequence completely analogously to the complex-oriented case (see \cite{HK01}). We have
$$\mathbb{ER}(n)^\star(BU(2))=\mathbb{ER}(n)^\star[[c_1, c_2]]$$
with $|c_i|=i(1+\alpha)$. We hat $c_2$ to produce $\widehat{c}_2=c_2y^2$ in degree $2(1-\lambda)+0\alpha$ which restricts to $c_2v_n^{2(2^n-1)} \in E(n)^*(BU(2))$. When we take fixed points to land in $ER(n)^*(BO(2))$ and map over to $ER(n)^*(BSO(2))=ER(n)^*(\cp)$, we claim this will produce a permanent cycle which lifts $\uub \in E(n)^*(\cp)$. That is, consider the following commutative diagram:
$$\xymatrix{
[BU(2), \ER(n)]^{C_2} \ar@{->}[d]^-{} \ar@{->}[r] & [BSO(2), \ER(n)]^{C_2} \ar@{=}[d] \\
[BU(2), \ER(n)] \ar@{=}[d] & [BSO(2), ER(n)] \ar@{->}[d]\\
[BU(2), E(n)] \ar@{->}[r] &[BSO(2), E(n)]}
$$
Here the two horizontal maps are induced by the inclusion $BSO(2) \longrightarrow BU(2)$. From the diagram, we conclude that the image of $\widehat{c}_2 \in \mathbb{ER}(n)^\star(BU(2))$ in $ER(n)^*(BSO(2))$ is a permanent cycle, whose representative on $E_1$ is given by mapping to the bottom right corner. Since $\widehat{c}_2$ restricts to $c_2v_n^{2(2^n-1)}$ in $E(n)^*(BU(2))$, it remains to show that the image of this class in $E(n)^*(BSO(2))$ is $\uub$. This follows from the homotopy commutativity of the following diagram:
$$\xymatrix{
BU(1) \ar@{->}[r]^-\Delta \ar@{->}[d]_{\simeq} & BU(1) \times BU(1) \ar@{->}[r]^{1 \times c} &BU(1) \times BU(1) \ar@{->}[d]^m \\
BSO(2) \ar@{->}[rr] &  & BU(2)
}$$
\end{proof}

\begin{remark}\label{remark:pontryagin} The above argument shows that, as an element of $E(n)^*(\cp)=E(n)^*(BSO(2))$, the class $\uub$ is in fact $v_n^{2(2^n-1)}$ times the first Pontryagin class in $E(n)$-cohomology and lifts to $ER(n)^*(\cp)$. We denote its lift by $\phat_1$ as in Theorem \ref{thm:SES}.
\end{remark}
Now that we have permanent cycles $(\uub)^l$ for $l \geq 0$, we can use these to form half of our basis for $E_1^{*, *}$. The other half of the basis will consist of classes $\uhat(\uub)^l$, $l \geq 0$. Since
$$(\uub)^l\equiv (-1)^l\uhat^{2l} \text{ mod } (\uhat^{2l+1})$$
and
$$\uhat (\uhat \ubar)^l \equiv (-1)^l\uhat^{2l+1} \text{ mod } (\uhat^{2l+2})$$
it follows that $\{\uhat^{\epsilon}(\uub)^l: \epsilon=0 \text{ or } 1,\text{ }l \geq 0\}$ clearly forms a topological basis for $E_1^{*, *}$ over $E(n)^*[x]$.

\section{Computing $E_2^{*, *}$} \label{sec:6}
To determine $d_1$ on this basis, it is necessary to distinguish between odd and even exponents of $v_n$, since $c(v_n^l)=(-1)^lv_n^l$. In what follows, recall that $\vhat_i$ are permanent cycles and note that $d_1(v_n^2)=0$. We have
\begin{align*}
d_1(v_n^{2p}(\uub)^l)&=0\\
d_1(v_n^{2p+1}(\uub)^l)&=2v_n^{2p-2^n}(\uub)^l x\\
d_1(v_n^{2p}(\uhat(\uub)^l))&=v_n^{2p-(2^n-1)}(\uhat-\ubar)(\uub)^lx\\
d_1(v_n^{2p+1}(\uhat(\uub)^l))&=v_n^{2p-2^n}(\uhat+\ubar)(\uub)^lx.
\end{align*}

Set $R=\ZZ_{(2)}[\vhat_1, \dots, \vhat_{n-1}, v_n^{\pm 2}][x]$ so that $E_1^{*, *}=(R \oplus v_nR)[[\uhat]]$. To analyze the image and kernel of $d_1$, we begin with a technical lemma concerning $\ubar$.

\begin{lemma}\label{lemma:ubar} $\ubar$ is in $R[[\uhat]]$.
\end{lemma}
\begin{proof} Consider $R$ as a submodule of $\ZZ_{(2)}[\vhat_1, \dots, \vhat_{n-1}, v_n^{\pm 1}][x]$ over the ring $\ZZ_{(2)}[\vhat_1, \dots, \vhat_{n-1}]$. Notice $\vhat_n=v_n^{-(2^n-1)^2+1}$ is in $R$. Thus, the coefficients of $\widehat{F}$, formed by hatting the coefficients of $F$, are also in $R$, as is $[2]_{\widehat{F}}(\uhat)$. We have
$$\uhat = [2]_{\widehat{F}}(\uhat)+_{\widehat{F}}\ubar.$$
Reducing modulo the submodule $R[[\uhat]]$, we have
$$0 \equiv 0 +_{\widehat{F}}\ubar \text{ mod }R[[\uhat]].$$
Thus, $\ubar \in R[[\uhat]]$.
\end{proof}

It follows from the formulas for $d_1$ above together with Lemma \ref{lemma:ubar} that $d_1$ interchanges classes in $R[[\uhat]]$ with classes in $v_nR[[\uhat]]$. We now describe a convenient (topological) basis for the kernel of $d_1$.

\begin{prop} \label{prop:uub}A basis for the kernel of $d_1$ over $R=\ZZ_{(2)}[\vhat_1, \dots, \vhat_{n-1}, v_n^{\pm 2}][x]$ is given by 
$$\{(\uhat \ubar)^l, v_n(\uhat-\ubar)(\uhat \ubar)^l: l \geq 0\}.$$
\end{prop}
\begin{proof} Let $f$ be in the kernel of $d_1$. We may write $f=f_e+v_nf_o$ with $f_e \in R[[\uhat]]$ and $v_nf_o \in v_nR[[\uhat]]$. Since $d_1$ interchanges classes in $R[[\uhat]]$ and $v_nR[[\uhat]]$, we must have both $d_1(f_e)=0$ and $d_1(v_nf_o)=0$. We will show that $f_e \in \text{span}\{(\uhat \ubar)^l\}$ and $v_nf_o\in \text{span}\{v_n(\uhat-\ubar)(\uhat \ubar)^l\}$. 

Let $f_e=\mu \uhat^j$ mod $(\uhat^{j+1})$ with $\mu \in R$. By Lemma \ref{lemma:ubarcong}, $\uhat^{\ast j} \equiv (-1)^j\uhat^j \text{ mod }(\uhat^{j+1})$. Since
$$d_1(f_e)\equiv \mu v_n^{1-2^n}(\uhat^j-\uhat^{\ast j}) \equiv \mu v_n^{1-2^n}(\uhat^j + (-1)^{j+1}\uhat^j) \text{ mod }(\uhat^{j+1})$$
must be zero, $j$ must be even. Then $f_e-(-1)^{\frac{j}{2}}\mu(\uub)^{\frac{j}{2}}$ is in $R[[\uhat]]$, in the kernel of $d_1$, and has $\uhat$-adic valuation strictly larger than that of $f_e$. Thus, $f_e$ may be approximated to any degree by polynomials in $R[\uub]$, which proves the claim for $f_e$.

We need to consider the first two terms in the case of $f_o$. Let $v_nf_o\equiv\mu v_n \uhat^j +\nu v_n \uhat^{j+1}$ mod $(\uhat^{j+2})$. Applying $d_1$ modulo $(\uhat^{j+1})$ shows that $j$ must now be odd. Next we apply $d_1$ modulo $(\uhat^{j+2})$. The congruences in Lemma \ref{lemma:ubarcong} give
\begin{align*}
\uhat^j+\uhat^{\ast j} &\equiv \vhat_1^j\uhat^{j+1} \text{ mod }(2, \uhat^{j+2})\\
\uhat^{j+1}+\uhat^{\ast j+1} &\equiv 0 \text{ mod }(2, \uhat^{j+2}).
\end{align*}
Thus,
\begin{align*}
0=d_1(v_nf_o)&\equiv v_n^{2-2^n}\mu (\uhat^j +\uhat^{\ast j})+v_n^{2-2^n}\nu (\uhat^{j+1}+\uhat^{\ast j+1})\\
&\equiv \vhat_1^j v_n^{2-2^n} \mu \uhat^{j+1}\text{ mod }(2, \uhat^{j+2}).\end{align*}
It follows that $\mu=2\gamma$ for some $\gamma \in R$. Then
$$\gamma v_n (\uhat - \ubar)(\uub)^{\frac{j-1}{2}}\equiv \mu \uhat^j \text{ mod}(\uhat^{j+1}).$$
Thus, $v_nf_o-(-1)^{\frac{j-1}{2}}\gamma v_n(\uhat - \ubar)(\uub)^{\frac{j-1}{2}}$ is in $v_nR[[\uhat]]$, is in ker$(d_1)$, and has $\uhat$-adic valuation strictly larger than that of $v_nf_o$. This shows that $f_o$ may be approximated to any degree by elements of $\text{span}\{v_n(\uhat-\ubar)(\uub)^l\}$, which proves the claim for $f_o$.

That the above set of elements is linearly independent follows from inspecting their leading terms.
\end{proof}

The next step is to relate the image of $d_1$ to its kernel. This consists of analyzing the class $\uhat+\ubar$ in terms of the above basis for the kernel.

\begin{lemma} \label{lemma:u+ubar} $\uhat+\ubar$ is in $\ZZ_{(2)}[\vhat_1, \dots, \vhat_{n-1}, v_n^{\pm 2}][[\uub]]$. In other words, there is a power series $\xi$ with coefficients in $\ZZ_{(2)}[\vhat_1, \dots, \vhat_{n-1}, v_n^{\pm 2}]$ such that $\uhat+\ubar=\xi(\uub)$.
\end{lemma}
\begin{proof} It follows from the proof of Proposition \ref{prop:uub} that $\uhat+\ubar \in R[[\uub]]$, as it shows that any class that is in $R[[\uhat]]$ and in ker$(d_1)$ is also in $R[[\uub]]$. Lemma \ref{lemma:ubar} shows that $\uhat+\ubar \in R[[\uhat]]$ and $c(\uhat+\ubar)=\uhat+\ubar$ shows it is in ker$(d_1)$. Since $x$ does not divide $\uhat+\ubar$, the coefficients of $\xi$ lie in $\ZZ_{(2)}[\vhat_1, \dots, \vhat_{n-1}, v_n^{\pm 2}]$.
\end{proof}

\begin{remark} In fact, something stronger is true. In the proof of Proposition \ref{prop:uub}, if we replace $R$ by the ring $\ZZ_{(2)}[\vhat_1, \dots, \vhat_n]$, the same argument applies to show that $\uhat+\ubar$ is in $\ZZ_{(2)}[\vhat_1, \dots, \vhat_n][[\uub]]$.
\end{remark}

We will say more about the power series expansion of $\uhat+\ubar$ in $\uub$ in section \ref{sec:11}. For now, we describe the $E_2$-page. We will present the result as a module over $E_2^{*, *}(pt)$. Recall that
\begin{align*}
E_2^{0, *}(pt)&=\ZZ_{(2)}[\vhat_1, \dots, \vhat_{n-1}, v_n^{\pm 2}]x^0\\
E_2^{s, *}(pt)&=\ZZ/2[\vhat_1, \dots, \vhat_{n-1}, v_n^{\pm 2}]x^s \ind \text{ for }s>0.
\end{align*}
 
We then have
\begin{thm}\label{thm:E2} The $E_2$-page is given by
\begin{align*} 
E_2^{0, *}&=E_2^{0, *}(pt)[[\uhat \ubar]]\{1, v_n(\uhat-\ubar)\}\\
E_2^{s, *}&=E_2^{s, *}(pt)[[\uhat \ubar]]/(\uhat + \ubar) \ind \text{ for }s>0.
\end{align*}
\end{thm}
\begin{proof} In the basis for the kernel given in Proposition \ref{prop:uub}, for $s>0$, the classes $v_n(\uhat-\ubar)(\uhat \ubar)^lx^s$ are targets of differentials as are the classes $2(\uub)^l x^s$. Thus, these classes only survive on the zero line. Away from the zero line, we just have $R[[\uhat \ubar]]x^s$ modulo the image of $d_1$.  Lemma \ref{lemma:u+ubar} shows that the ideal generated by $(\uhat+\ubar)x$ in $E_1^{*, *}$ is contained in $R[[\uub]][x]$. This proves the theorem.

\end{proof}


\section{The image of the norm} \label{sec:7}
We now find ourselves in a very nice place. We know how the differentials act on the coefficients, and we have a large collection of permanent cycles. The remaining permanent cycles live on the zero line and the next step is to find representatives for them in $ER(n)^*(\cp)$. We do this using the norm described by Proposition \ref{thm:norm} in the appendix. We let $N_*$ denote the map
$$N_*: E(n)^*(\cp) \longrightarrow ER(n)^*(\cp)$$
and $\N_*$ denote the map resulting from postcomposing $N_*$ with the inclusion of fixed points map $ER(n)^*(\cp) \rightarrow E(n)^*(\cp)$,
$$\N_*: E(n)^*(\cp) \longrightarrow ER(n)^*(\cp) \longrightarrow E(n)^*(\cp).$$
Thus, for any $z \in E(n)^*(\cp)$, $\N_*(z)$ is a permanent cycle on the zero line represented in $ER(n)^*(\cp)$ by $N_*(z)$. From the appendix, we have $\N_*(z)=z+c(z)$. For any $w$ such that $c(w)=w$, we have $\N_*(wz)=w\N_*(z)$. Let $S=\ZZ_{(2)}[\vhat_1, \dots, \vhat_{n-1}, v_n^{\pm 2}]$ (so that $R$ above is $S[x]$). As a module over $S[[\uub]]=\ZZ_{(2)}[\vhat_1, \dots, \vhat_{n-1}, v_n^{\pm 2}][[\uub]]$, we may write
$$E(n)^*(\cp)=S[[\uub]]\{1, v_n, \uhat, v_n\uhat\}.$$
Since $c$ fixes $S[[\uub]]$, $\N_*$ is a map of modules over $S[[\uub]]$. We restrict $\N_*$ to the submodule $S[[\uub]]\{\uhat, v_n\uhat\} \subset E(n)^*(\cp)$ and let $\imNr$ denote the image. (We restrict to the submodule because we do not want the coefficients, in particular 2, to be in $\imNr$. This is because we will mod out by $\imN$ later, and we will want multiplication by 2 to be injective on the quotient.) We have
\begin{align*}\N_*(\uhat)&=\uhat+\ubar \\
\N_*(v_n\uhat)&=v_n(\uhat-\ubar)\end{align*}
so
$$\imNr=S[[\uub]]\{\uhat+\ubar, v_n(\uhat-\ubar)\}.$$
This is a submodule of $E_1^{0, *}$, and furthermore, since elements of $\imNr$ are permanent cycles, it is contained in ker$(d_1)$. Since no differentials have their targets in the zero line, it follows that $\imNr$ is a submodule of $E_2^{0, *}$. We have a short exact sequence
$$0 \longrightarrow \imNr \longrightarrow E_2^{*, *} \longrightarrow \widetilde{E}_2^{*, *} \longrightarrow 0$$
where $\widetilde{E}_2^{*, *}$ is by definition the quotient. Since all differentials on $\imNr$ are zero, we may further view it as a sub-spectral sequence. Furthermore, since $\imNr$ injects into $E_r^{*, *}$ at each stage, it follows that the above short exact sequence is in fact a short exact sequence of spectral sequences. From Theorem \ref{thm:E2} we conclude
$$\widetilde{E}_2^{*, *}=\frac{E_2^{*, *}(pt)[[\uub]]}{(\uhat+\ubar)}.$$
In the short exact sequence above $\imNr$ collapses immediately, so it remains to compute the spectral sequence $\widetilde{E}_2^{*, *}$.

 \section{Landweber flatness} \label{sec:8}
 Let $\widehat{E}(n)^*=\ZZ_{(2)}[\vhat_1, \dots, \vhat_{n-1}, \vhat_n^{\pm 1}]$. There is an isomorphism of rings (but not graded rings) between $E(n)^*$ and $\widehat{E}(n)^*$ sending $v_k$ to $\vhat_k$. $\widehat{E}(n)^*$ consists entirely of permanent cycles. Thus $\widehat{E}(n)^* \subset ER(n)*$ and $ER(n)^*$ is a module over $\widehat{E}(n)^*$. We may view $E_r^{*, *}(pt)$ as a spectral sequence of $\widehat{E}(n)^*$-modules. The $E_2$-page of the spectral sequence of interest, $\widetilde{E}_2^{*, *}$, may be written as
 $$\widetilde{E}_2^{*, *}=E_2^{*, *}(\text{pt})[[\uhat\ubar]]/(\uhat+\ubar)=E_2^{*, *}(pt) \otimes_{\widehat{E}(n)^*} \widehat{E}(n)^*[[\uub]]/(\uhat+\ubar).$$
 The right hand coordinate of the tensor product consists entirely of permanent cycles, so we will be done if we can show that we can commute taking homology past the tensor product at each stage. That is, we need to know that tensoring with $\widehat{E}(n)^*[[\uub]]/(\uhat+\ubar)$ over $\widehat{E}(n)^*$ is exact. This would be true if $\widehat{E}(n)^*[[\uub]]/(\uhat+\ubar)$ were flat over $\widehat{E}(n)^*$, but we can in fact show that a weaker condition holds.
 
 If we identify $E(n)^*$ with $\widehat{E}(n)^*$ as above, then we may view $E_r^{*, *}(pt)$ as a spectral sequence of $E(n)^*$-modules. Starting with this observation, it is shown in \cite{KW14} that $E_r^{*, *}(pt)$ in fact lives in the category of $E(n)_*E(n)$-comodules that are finitely presented as $E(n)^*$-modules. 
 
 Thus, to solve our problem we only need to show that $\widehat{E}(n)^*[[\uub]]/(\uhat+\ubar)$ is flat on the category of finitely presented $E(n)_*E(n)$-comodules, i.e. that it is \emph{Landweber flat}. 
 
 In \cite{HS05}, Hovey and Strickland prove an $E(n)$-version of the Landweber filtration theorem. We state and prove an $E(n)$-exact functor theorem, which follows as a corollary of Hovey and Strickland's work. To be consistent with the literature, we prove the result for $E(n)_*$-modules, keeping in mind that $E(n)_*$ is formally isomorphic (as rings but not graded rings) to $E(n)^*$ so everything below holds for $E(n)^*$-modules as well.
 
 \begin{prop}\label{prop:land} Let $M$ be an $E(n)_*$-module. The functor $(-)\otimes_{E(n)_*}M$ is exact on the category of $E(n)_*E(n)$-comodules that are finitely presented as $E(n)_*$-modules if and only if for each $k\geq 0$ multiplication by $v_k$ is monic on $M/(v_0, \dots, v_{k-1})M$, i.e. $(v_0, v_1, v_2, \dots)$ is a regular sequence on $M$.
 \end{prop}
 \begin{remark} Since $E(n)_*$ is height $n$ in the sense of Hovey-Strickland, there is only something to check for $0 \leq k \leq n$. For $k>n$, $M/(v_0, \dots, v_{k})=0$ so multiplication by $v_k$ is trivially monic. 
 \end{remark}
 \begin{proof} We follow Landweber's original proof over $MU_*$ in \cite{Lan76}. Applying $(-)\otimes_{E(n)_*}M$ to the sequence
 $$\begin{CD}0 @>>> E(n)_* @>p>> E(n)_* @>>> E(n)_*/(p) @>>>0 \end{CD}$$
 shows that $p: M \longrightarrow M$ is monic if and only if $\Tor_1^{E(n)_*}(E(n)_*/(p), M)=0$. 
 For $k>0$, applying $(-)\otimes_{E(n)_*}M$ to the sequence
 $$\small \begin{CD}0 @>>> E(n)_*/(v_0, \dots, v_{k-1}) @>v_k>> E(n)_*/(v_0, \dots, v_{k-1}) @>>> E(n)_*/(v_0, \dots, v_n) @>>>0\end{CD} $$
 shows that multiplication by $v_k$ is monic if and only if
 $$\Tor_1^{E(n)_*}(E(n)_*/(v_0, \dots, v_{k-1}), M) \longrightarrow \Tor_1^{E(n)_*}(E(n)_*/(v_0, \dots, v_k), M)$$
 is surjective. It follows that multiplication by $v_k$ is monic on $M/(v_0, \dots, v_{k-1})$ for all $k$ if and only if $\Tor_1^{E(n)_*}(E(n)_*/(v_0, \dots, v_k), M)$ is zero for all $k$. In \cite{HS05} it is shown that every $E(n)_*E(n)$-comodule $N$ that is finitely presented over $E(n)_*$ admits a finite filtration by subcomodules
 $$0=N_0 \subseteq N_1 \subseteq \dots \subseteq N_s=N$$
 for some $s$ with $N_r/N_{r-1} \equiv \Sigma^tE(n)_*/(v_0, \dots, v_j)$ for some $j \leq n$ and some $t$, both depending on $r$. In view of this, $\Tor_1^{E(n)_*}(E(n)_*/(v_0, \dots, v_k), M)=0$ for all $k$ is equivalent to $\Tor_1^{E(n)_*}(N, M)=0$ for all finitely presented $E(n)_*E(n)$-comodules, $N$. Finally, this is equivalent to $(-)\otimes_{E(n)_*}M$ being an exact functor on the category of $E(n)_*E(n)$-comodules finitely presented over $E(n)_*$.
 \end{proof}
 
 We now show that $M=\widehat{E}(n)^*[[\uub]]/(\uhat+\ubar)$ satisfies the algebraic criterion given above.
 
 \begin{lemma} $(\vhat_0, \dots, \vhat_{n-1}, \vhat_n)$ is a regular sequence in $\widehat{E}(n)^*[[\uub]]/(\uhat+\ubar)$.
 \end{lemma}
 \begin{proof} Recall our notation $I_k=(\vhat_0, \vhat_1, \dots, \vhat_{k-1})$ and $I_0=(0)$. Suppose $f(\uub) \in \widehat{E}(n)^*[[\uub]]/(I_k, \uhat+\ubar)$ is such that $\vhat_k f(\uub)=0$. Then
 $$\vhat_k f(\uub)=g(\uub)(\uhat+\ubar) \text{  mod }I_k.$$
Further modding out by $\vhat_k$, we have
$$0=g(\uub)(\uhat+\ubar) \text{ mod }I_{k+1}.$$
By Lemma \ref{lemma:ubarcong}, we have that 
$$\uhat+\ubar=\vhat_{k+1}(\uhat \ubar)^{2^{k+1}} \text{ mod } (I_{k+1}, (\uub)^{2^{k+1}+1})$$
which means $\uhat+\ubar \neq 0$ mod $I_{k+1}$. Since $I_{k+1}$ is prime in $\widehat{E}(n)^*$, it follows that $g(\uub)=0$ mod $I_{k+1}$. Then $g(\uub)=\vhat_k h(\uub)$ mod $I_k$ for some $h(\uub)$. Hence,
$$f(\uub)=h(\uub)(\uhat+\ubar) \text{ mod }I_k$$
so that $f(\uub)=0$ in $\widehat{E}(n)^*[[\uub]]/(I_k, \uhat+\ubar)$. Thus, multiplication by $\vhat_k$ is injective.
 \end{proof}
 
 Thus, $M$ is Landweber flat. That is, tensoring with $M$ over $\widehat{E}(n)^*$ is an exact functor on the category of finitely presented $E(n)_*E(n)$-comodules. Since $E_r^{*, *}(pt)$ lives in this category, we may commute homology past the tensor product at each stage of $\widetilde{E}_{r, }^{*, *}=E_r^{*, *}(pt) \otimes_{\widehat{E}(n)^*} M$. Furthermore, since $M$ consists entirely of permanent cycles by Proposition \ref{thm:uubar}, the entire spectral sequence can be evaluated on the coefficients. In other words, Theorem 4.3 in \cite{KW14} applies to show that $\widetilde{E}_{r}^{*, *}$ is isomorphic to $( E_r^{*, *}(pt)\otimes_{\widehat{E}(n)^*}M , d_r\otimes_{\widehat{E}(n)^*}\text{id}_M)$ as spectral sequences of $ER(n)^*$-modules and converges to $ER(n)^* \otimes_{\widehat{E}(n)^*}M$. We conclude
 
 \begin{prop}
 $$\widetilde{E}_{\infty}^{*, *}=E_\infty^{*, *}(pt) \otimes_{\widehat{E}(n)^*} M=E_\infty^{*, *}(pt)[[\uub]]/(\uhat+\ubar)$$
 \end{prop}


\section{The $E_\infty$-page} \label{sec:9}
We will now put all of the pieces together. Let us return to the short exact sequence of $E_2$-terms we had in Section \ref{sec:7}. 
$$0 \longrightarrow \imNr \longrightarrow E_2^{*, *} \longrightarrow \widetilde{E}_2^{*, *} \longrightarrow 0$$

\begin{lemma} Upon taking homology the induced long exact sequence collapses into short exact sequences at each stage. In particular, we have a short exact sequence 
$$0 \longrightarrow \imNr \longrightarrow E_{\infty}^{*, *} \longrightarrow E_\infty^{*, *}(\text{pt})[[\uhat\ubar]]/(\uhat+\ubar)\longrightarrow 0.$$
\end{lemma}
\begin{proof}
Since the connecting homomorphism of the long exact sequence must increase filtration degree (because the differentials do), yet $\imNr$ is concentrated in filtration degree zero, the connecting homomorphism is zero at each stage. Thus we have a short exact sequence at $E_\infty$. Recall that the left hand spectral sequence collapses immediately. The right hand spectral sequence was computed in the previous section.
\end{proof}

We analyze this short exact sequence, starting away from the zero line. In strictly positive filtration degree, the left hand term is zero and so we have
$$E_{\infty}^{s, *}=E_\infty^{s, *}(pt)[[\uub]]/(\uhat+\ubar) \ind \text{ for }s>0.$$

The zero line is more involved. We begin by giving a ``polite'' answer.

\begin{prop} The zero-line $E_\infty^{0, *}$ injects into its rationalization, and the rationalization may be computed as the algebraic invariants of the rationalization of $E_1^{0,*}$:
$$\xymatrix{
E_\infty^{0, *} \ar@{^{(}->}[r] & E_\infty^{0, *} \otimes \mathbb{Q} \ar@{->}[r]^-\cong & (E_1^{0, *} \otimes \mathbb{Q})^{C_2}
}$$
\end{prop}
\begin{proof}Since no classes on the zero line are targets of differentials, $E_\infty^{0, *}$ contains no torsion and so injects into its rationalization. $E_2^{0, *}$ is exactly the invariants in $E_1^{0, *}$. Away from the zero line, everything is 2-torsion from $E_2$ onward. Thus, for any class in $E_r^{0, *}$, twice it is in the kernel of $d_r$. Since $d_{2^{n+1}-1}$ is the last possible differential, we have that $2^{2^{n+1}-1}$ times any class on $E_2^{0, *}$ survives to $E_\infty$. After we rationalize, the isomorphism
$$E_\infty^{0, *} \otimes \mathbb{Q} \cong (E_1^{0, *} \otimes \mathbb{Q})^{C_2}$$
follows.
\end{proof}

We now describe the zero line explicitly. Recall that
$$E_2^{0, *}(pt)=\ZZ_{(2)}[\vhat_1, \dots, \vhat_{n-1}, v_n^{\pm 2}]$$
$$E_\infty^{0, *}(pt)=\ZZ_{(2)}[\vhat_1 , \dots, \vhat_{n-1}, v_n^{\pm 2^{n+1}}] \bigoplus_{0<i<n+1} (\vhat_0, \dots, \vhat_{i-1})\ZZ_{(2)}[\vhat_1 , \dots, \vhat_{n-1}, v_n^{\pm 2^{i+1}}]v_n^{2^{i}}$$

We have a short exact sequence of $E_2$-terms
$$0 \longrightarrow \imNr \longrightarrow E_2^{0, *} \longrightarrow \widetilde{E}_2^{0, *} \longrightarrow 0$$
where, as modules over $E_2^{0, *}(pt)$, we have
$$\imNr=E_2^{0, *}(pt)[[\uub]]\{\uhat+\ubar, v_n(\uhat-\ubar)\}$$
$$\widetilde{E}_2^{0, *}=E_2^{0, *}(pt)[[\uub]]/(\uhat+\ubar)$$
and
$$E_2^{0, *}=E_2^{0, *}(pt)[[\uub]]\{1, v_n(\uhat-\ubar)\}.$$
Recall that $\uhat+\ubar=\xi(\uub)\cdot 1$. When we pass to the $E_\infty$-page, we have, as a module over $E_\infty^{0, *}(pt)$, 
$$\imNr=E_\infty^{0, *}(pt)[[\uub]]\{v_n^{2p}(\uhat+\ubar), v_n^{2p+1}(\uhat-\ubar)\}/J, \ind 0 \leq p < 2^n$$
where $J$ is the submodule generated by the following relations (which come from writing $E_2^{0, *}(pt)$ as a module over $E_\infty^{0, *}(pt)$). Write $2p=2^{i+1}m+2^i$. For $0 \leq j < i$, we have
$$\vhat_j \cdot [v_n^{2p}(\uhat+\ubar)]=(\vhat_jv_n^{2p}) \cdot (\uhat+\ubar)$$
$$\vhat_j \cdot [v_n^{2p+1}(\uhat-\ubar)]=(\vhat_jv_n^{2p}) \cdot v_n(\uhat-\ubar).$$

Over $E_{\infty}^{0, *}(pt)$, we also have
$$\widetilde{E}_\infty^{0, *}=E_\infty^{0, *}(pt)[[\uub]]/(\uhat+\ubar).$$
Thus, we conclude
\begin{thm} \label{thm:einfty}As a module over $E_\infty^{0, *}(pt)$, we have
$$E_\infty^{0, *}=\frac{E_\infty^{0, *}(pt)[[\uub]]\{1, v_n^{2p}(\uhat+\ubar), v_n^{2p+1}(\uhat-\ubar)\}}{K}, \ind 0 \leq p <2^n$$
where $K$ encodes the relations generating $J$ in $\imNr$ over $E_\infty^{0, *}$ above together with the relation $v_n^0(\uhat+\ubar)=\xi(\uub)\cdot 1$.
\end{thm}

\section{Extension problems and multiplicative structure} \label{sec:10}
Most of the hard work in solving extension problems is already done by Propositions \ref{thm:uubar} and \ref{thm:norm} as they provide canonical lifts to $ER(n)^*(\cp)$ of our generators of $E_\infty^{*, *}$. Proposition \ref{thm:uubar} produces a class $\phat_1 \in ER(n)^*(\cp)$ whose image in $E(n)^*(\cp)$ is $\uub$. On the zero line, we also have classes $v_n^{k}(\uhat+(-1)^k\ubar)$ which are the images under the norm of classes $v_n^k\uhat \in E(n)^*(\cp)$. Since the norm factors through the map $ER(n)^*(\cp) \longrightarrow E(n)^*(\cp)$, these classes have canonical lifts in $ER(n)^*(\cp)$ as well.

In Lemma \ref{lemma:u+ubar} we showed that in $E(n)^*(\cp)$, $\uhat+\ubar$ may be written as a power series $\xi(\uub)$, with the coefficients of $\xi$ in $\widehat{E}(n)^*$. We have independently constructed lifts $N_*(\uhat)$ of $\uhat+\ubar$ and $\phat_1$ of $\uub$, so we must verify this equality lifts to $ER(n)^*(\cp)$. This turns out to be true for degree reasons: 

\begin{lemma}\label{lemma:twolifts} $N_*(\uhat)=\xi(\phat_1)$ in $ER(n)^*(\cp)$.
\end{lemma}
\begin{proof} The two classes have the same image in $E(n)^*(\cp)$, so their difference is a multiple of $x$. If $N_*(\uhat)-\xi(\phat_1) \neq 0$ in $ER(n)^*(\cp)$, let $r$ be the maximal power of $x$ that divides $N_*(\uhat)-\xi(\phat_1)$.  Suppose $r \geq 1$. Then $N_*(\uhat)-\xi(\phat_1)$ is represented by a nonzero class $z \in E_\infty^{r, 1-\lambda+r}$. Since $r \geq 1$, we have $2^j-1 <r \leq 2^{j+1}-1$. Since $E_\infty^{r, *}=0$ for $r\geq 2^{n+1}$, we must have $j <n+1$. By inspection of degrees, we have that $E_\infty^{r, l} =0$ unless $l =0$ mod $2^{j+1}$. Then
$$1-\lambda+r =0 \text{ mod }2^{j+1}.$$
Since $1-\lambda =-2^{n+2}(2^{n-1}-1)$, it follows that
$$r=0 \text{ mod }2^{j+1}$$
which is impossible. Thus, $N_*(\uhat)-\xi(\phat_1)=0$ in $ER(n)^*(\cp)$.
\end{proof}

Before we go on to solve the remaining extension problems, we pause to remark on the significance of the class $N_*(\uhat)=\xi(\phat_1)$ above which appears in the denominator of the right hand term of the short exact sequence of Theorem \ref{thm:SES}.

\begin{remark}\label{remark:uhat+ubar} $N_*(\uhat)$ is a canonical class in $ER(n)^*(\cp)$ in the following sense. The ``hatted" orientation $\uhat$ induces a map
$$\xymatrix{B(U(1) \rtimes C_2)_+=\cp_+ \wedge_{C_2}EC_{2_+} \ar@{->}[r]^-{\uhat \wedge 1} & \Sigma^{1-\lambda}\ER(n)\wedge_{C_2}EC_{2_+}}.$$
Postcomposing with the Adams isomorphism $\ER(n)\wedge_{C_2}EC_{2_+} \simeq \ER(n)^{hC_2}=ER(n)$ and precomposing with $BU(1)_+ \rightarrow B(U(1) \rtimes C_2)_+$ yields the class $N_*(\uhat)$ in the $ER(n)$-cohomology of $BU(1)=\cp$ whose image in $E(n)^*(\cp)$ is the $\uhat+\ubar$ described above. This description also gives an alternate argument that $N_*(\uhat)$ is a power series on $\phat_1$ as follows. Identifying $B(U(1) \rtimes C_2)$ above with $BO(2)$, it is shown in \cite{KW14} that $ER(n)^*(BO(2))$ is the quotient of a power series ring over $ER(n)^*$ on two classes, $\chat_1$ and $\chat_2$. The above description of $N_*(\uhat)\in ER(n)^*(\cp)$ shows that it is the image of a class in $ER(n)^*(BO(2))$, i.e. some power series in $\chat_1$ and $\chat_2$. Identifying $BU(1)$ with $BSO(2)$, it can be shown that the map $ER(n)^*(BO(2)) \longrightarrow ER(n)^*(BSO(2))$ above sends $\chat_1$ to zero and $\chat_2$ to $\phat_1$ (see Proposition \ref{thm:uubar}). It follows that $N_*(\uhat)$ is a power series on $\phat_1$ over $ER(n)^*$. Note that this argument does not identify the power series explicitly---to do that, we still need Lemmas \ref{lemma:u+ubar} and \ref{lemma:twolifts}.
\end{remark}

With canonically determined lifts in hand, we now solve all extension problems. The relations of Theorem \ref{thm:einfty} need to be lifted to $ER(n)^*(\cp)$. Additionally, we describe how classes in $\imnr$ multiply together. In both cases, we use Proposition \ref{thm:norm} in the appendix. In $1(a)$ and $(b)$ of the following lemma, recall that classes in $E(n)^*$ of the form 
$$\vhat_jv_n^{2^{i+1}m+2^i} \ind \text{ with }0 \leq j < i$$ 
are permanent cycles and lift to classes of the same name in $ER(n)^*$.

\begin{lemma}\label{lemma:mult} The following relations hold in $ER(n)^*(\cp)$:

\begin{enumerate}[1.]
\item Write $2p=2^{i+1}m+2^i$ and suppose $0 \leq j < i$. Then, as a module over $ER(n)^*$, $ER(n)^*(\cp)$ satisfies
\begin{enumerate}[(a)]
\item $\vhat_j\cdot N_*(v_n^{2p}\uhat)=(\vhat_jv_n^{2p}) \cdot N_*(\uhat)$
\item $\vhat_j \cdot N_*(v_n^{2p+1}\uhat)=(\vhat_jv_n^{2p}) \cdot N_*(v_n\uhat)$
\item $x \cdot N_*(v_n^{2p}\uhat)=0$
\item $x \cdot N_*(v_n^{2p+1}\uhat)=0$.
\end{enumerate}
\item Let $0 \leq 2p, 2l < 2^{n+1}$ and write $2p+2l=2^{n+1}q+2r$ with $0 \leq 2r<2^{n+1}$ and $q=0, 1$. As an algebra over $ER(n)^*$, $ER(n)^*(\cp)$ satisfies
\begin{enumerate}[(a)]
\item $N_*(v_n^{2p}\uhat)N_*(v_n^{2l}\uhat)=v_n^{{2^{n+1}}q}N_*(v_n^{2r}\uhat)N_*(\uhat)$
\item $N_*(v_n^{2p+1}\uhat)N_*(v_n^{2l}\uhat)=v_n^{{2^{n+1}}q}N_*(v_n^{2r+1}\uhat)N_*(\uhat)$
\item $N_*(v_n^{2p+1}\uhat)N_*(v_n^{2l+1}\uhat)=v_n^{{2^{n+1}}q}\left(N_*(v_n^{2r+2}\uhat)N_*(\uhat)-4v_n^{2r+2}\phat_1\right)$.
\end{enumerate}
\end{enumerate}
\end{lemma}

\begin{proof} The key facts are that $N_*$ is a map of modules over $ER(n)^*(\cp)$ and that the $ER(n)^*(\cp)$-module (really, algebra) structure on $E(n)^*(\cp)$ comes from the quotient-by-$x$ map $ER(n)^*(\cp) \longrightarrow E(n)^*(\cp)$. The module structure is described by the following diagram:
$$\xymatrix{
ER(n)^*(\cp) \otimes E(n)^*(\cp) \ar@{->}[d] \ar@{->}[r]^-{1 \otimes N_*} & ER(n)^*(\cp) \otimes ER(n)^*(\cp) \ar@{->}[d]\\
E(n)^*(\cp) \ar@{->}[r]^-{N_*} & ER(n)^*(\cp)
}$$
We prove $1(a), 1(c)$ and $2(a)$; the other relations are proved similarly. For $1(a)$, note that in the upper left corner, the classes $\vhat_j \otimes v_n^{2p}\uhat$ and $\vhat_jv_n^{2p}\otimes \uhat$ have as their images in the bottom right corner the classes $\vhat_j\cdot N_*(v_n^{2p}\uhat)$ and $(\vhat_jv_n^{2p})\cdot N_*(\uhat)$, respectively. But both $\vhat_j\otimes v_n^{2p}\uhat$ and $\vhat_jv_n^{2p}\otimes \uhat$ map to the same class, $\vhat_jv_n^{2p}\uhat$, in the bottom left corner, which proves $1(a)$. In $1(c)$, $x \otimes v_n^{2p}\uhat$ maps to zero in the bottom left corner (since $x$ maps to zero in $E(n)^*$); thus, $x \cdot N_*(v_n^{2p}\uhat)=0$. Finally, for $2(a)$, note that the classes $v_n^{2p}\uhat \otimes N_*(v_n^{2l}\uhat)$ and $v_n^{2(p+l)}\uhat \otimes N_*(\uhat)$ map to the same class under the left vertical map. Thus, their images in $ER(n)^*(\cp)$ are equal. But these are exactly $N_*(v_n^{2p}\uhat)N_*(v_n^{2l}\uhat)$ and $N_*(v_n^{2(p+l)}\uhat)N_*(\uhat)$, respectively. For any $\alpha \in E(n)^*(\cp)$ which admits a lift to $ER(n)^*(\cp)$ and any $z \in E(n)^*(\cp)$, the above diagram shows that $N_*(\alpha z)=\alpha N_*(z)$. Applying this to $\alpha=v_n^{2^{n+1}q}$ and $z=v_n^{2r}\uhat$ completes the proof of $2(a)$.
\end{proof}

We now prove Theorem \ref{thm:SES} as stated in the introduction. \\

\emph{Proof of Theorem \ref{thm:SES}.} $ER(n)^*(\cp)$ is (topologically) generated over $ER(n)^*$ by $\imnr$ and the classes $\phat_1^l$. To see the isomorphism (of $ER(n)^*$-modules)
$$ER(n)^*(\cp)/\imnr \cong ER(n)^*[[\phat_1]]/(\xi(\phat_1))$$
we must show that the intersection of $\imnr$ and the submodule of $ER(n)^*(\cp)$ generated over $ER(n)^*$ by $\{\phat_1^l| l \geq 0\}$ is precisely the ideal of $ER(n)^*(\cp)$ generated by $\xi(\phat_1)$. Clearly, $(\xi(\phat_1))$ is in this intersection. We now prove the reverse inclusion. First recall that the domain of $\Nr$ is by definition the submodule 
$${\ZZ_{(2)}[\vhat_1, \dots, \vhat_{n-1}, v_n^{\pm 2}][[\uub]]\{\uhat, v_n\uhat\}}$$
of $E(n)^*(\cp)$. Choose some
$$z=\sum_i (a_i\uhat+b_iv_n\uhat)(\uub)^i \ind \text{ with } a_i, b_i \in \ZZ_{(2)}[\vhat_1, \dots, \vhat_{n-1}, v_n^{\pm 2}]$$
so that (since $N_*$ is a map of modules over $ER(n)^*(\cp)$)
\begin{align}\label{eqn:1}N_*(z)&=\sum_i (N_*(a_i\uhat)+N_*(b_iv_n\uhat))\phat_1^i.\end{align}
Suppose that $N_*(z)$ is also in span$\{\phat_1^l|l \geq 0\}$ over $ER(n)^*$, i.e. that
\begin{align}\label{eqn:2}N_*(z)&=\sum_i\lambda_i\phat_1^i \ind \text{ with }\lambda_i \in ER(n)^*.\end{align}
We claim that $\xi(\phat_1)=N_*(\uhat)$ divides $N_*(z)$ in $ER(n)^*(\cp)$. This will follow from two claims: for all $i$, (i) $b_i=0$ and (ii) $a_i$ is in the subalgebra $E_\infty^{0, *}(pt)$ of $\ZZ_{(2)}[\vhat_1, \dots, \vhat_{n-1}, v_n^{\pm 2}]$. Together, they imply that $N_*(z)=\sum_i N_*(a_i\uhat)\phat_1^i$ with each $a_i$ having a representative in $ER(n)^*$. Since $N_*$ is a map or $ER^*(\cp)$-modules, it follows that $N_*(a_i\uhat)=a_iN_*(\uhat)$ and so $N_*(\uhat)$ divides $N_*(z)$ in $ER(n)^*(\cp)$. 

To prove both claims, we map into $E(n)^*(\cp)$. Then (\ref{eqn:1}) becomes
\begin{align}\label{eqn:3}\N_*(z)&=\sum_i(a_i(\uhat+\ubar)+b_iv_n(\uhat-\ubar))(\uub)^i\end{align}
and (\ref{eqn:2}) becomes
\begin{align}\label{eqn:4}\N_*(z)&=\sum_i\lambda_i(\uub)^i\end{align}
with $\lambda_i$ now in $E_\infty^{0, *}(pt)$ (by abuse of notation, we denote the image of $\lambda_i$ in $E_\infty^{0, *}(pt)$ by the same name). Recall that $\uhat-\ubar=2\uhat+\dots$ and $\uhat+\ubar$ is the span of $\{(\uub)^l|l \geq 0\}$ over $E_\infty^{0, *}(pt)$ (since $\vhat_i \in E_\infty^{0, *}(pt)$ for all $i$). The collection $\{(\uub)^l, \uhat(\uub)^l\}$ is clearly linearly independent over $E(n)^*$. Hence, from inspecting the right hand sides of (\ref{eqn:3}) and (\ref{eqn:4}), it follows that $b_i=0$ for all $i$.

To prove the second claim, suppose that $i_1$ is the first index such that $a_{i_1} \notin E_\infty^{0, *}(pt)$. Let $l$ be the maximal such that $a_{i_1}=f(v_n^{2^l})$ with the coefficients of $f$ in $\ZZ_{(2)}[\vhat_1, \dots, \vhat_{n-1}]$. Note that $f$ cannot be a constant polynomial since $a_{i_1}$ is not in $E_\infty^{0, *}$ by assumption. Then inspection of $E_\infty^{0, *}(pt)$ shows that $\vhat_ka_{i_1} \in E_\infty^{0, *}(pt)$ for $0 \leq k < l$ and $\vhat_ka_{i_1} \notin E_\infty^{0, *}(pt)$ for $k\geq l$. Recall that
$$\uhat+\ubar \equiv \vhat_l(\uub)^{2^l} + \dots \text{ mod }(2, \vhat_1, \dots, \vhat_{l-1}).$$
Since we know that
$$\sum_i a_i(\uhat+\ubar)(\uub)^i=\sum_i\lambda_i(\uub)^i \ind \text{ with }\lambda_i \in E_\infty^{0, *}$$
it follows from looking at lowest degree terms mod $(2, \vhat_1, \dots, \vhat_{l-1})$ that $\vhat_la_{i_1}$ is in $E_\infty^{0, *}$, a contradiction. Thus, $a_i \in E_\infty^{0, *}$ for all $i$. This proves Theorem \ref{thm:SES}.

\hfill $\Box$

Theorem \ref{thm:SES} is the nice form of the answer, but we have in fact shown something stronger. The following theorem presents $ER(n)^*(\cp)$ explicitly as an algebra over $ER(n)^*$.

\begin{thm} \label{thm:final} 
$$ER(n)^*(\cp)=ER(n)^*(pt)[[\phat_1, N_*(v_n^j\uhat)]]/K$$
where $0\leq j<2^{n+1}$ and $K$ is the ideal generated by the relation $N_*(\uhat)=\xi(\phat_1)$ of Lemma \ref{lemma:twolifts} together with the relations of Lemma \ref{lemma:mult}
\end{thm}
\begin{proof} This is a consequence of Lemmas \ref{lemma:twolifts} and \ref{lemma:mult} together with the description of $E_\infty^{*, *}$ in section \ref{sec:9}.
\end{proof}

\begin{remark}\label{remark:KO} We conclude by returning to the case $n=1$. It is somewhat degenerate, as $\N_*(\uhat)=\uhat+\ubar=-\uub$ in $KU^*(\cp)$, so $N_*(\uhat)=-\phat_1$. It follows that $x \cdot \phat_1=0$ and powers of $\phat_1$ do not generate any 2-torsion. This is \emph{not} true for $n>1$. In general, $\phat_1$ supports higher powers of $x$ up to and including $x^{2^{n+1}-2}$. In light of this, when $n=1$, $\phat_1$ is redundant and it suffices to take $\{1, N_*(\uhat), N_*(v_1\uhat), N_*(v_1^2\uhat), N_*(v_1^3\uhat)\}$ as a set of algebra generators with $x \cdot N_*(v_1^k \uhat)=0$ for all $k$. The relations come from Lemma \ref{lemma:mult} and our answer matches exactly the answer in Corollary 2.13 of \cite{Yam07}. \end{remark}

 \section{Completing at $I$} \label{sec:11}
 
 We obtain an especially nice form of the answer if we complete at $I:=I_n=(2, \vhat_1, \dots, \vhat_{n-1})$. It turns out that the right hand side of the short exact sequence of Theorem \ref{thm:SES} is free after completion. To see this, we need a technical lemma concerning $\uhat+\ubar$. 
 
\begin{lemma} The following congruence holds in $\widehat{E}(n)[[\uub]]$:
$$\uhat+\ubar \equiv \vhat_n(\uub)^{2^{n-1}} \text{ mod }((\uub)^{2^{n-1}+1}, I)$$
Since $\widehat{E}(n)\subset ER(n)^*$, this lifts to
$$\xi(\phat_1)\equiv \vhat_n\phat_1^{2^{n-1}} \text{ mod }(\phat_1^{2^{n-1}+1}, I)$$
in $ER(n)^*[[\phat_1]]$.
\end{lemma}
\begin{proof} By Lemma \ref{lemma:u+ubar}, $\uhat+\ubar \in \widehat{E}(n)[[\uub]]$. By Lemma \ref{lemma:ubarcong}, 
$$\uhat+\ubar \equiv \vhat_n\uhat^{2^n} \text{ mod }(\uhat^{2^n+1}, I).$$
It follows that
$$\uhat+\ubar\equiv \vhat_n(\uub)^{2^{n-1}} \text{ mod }((\uub)^{2^{n-1}+1}, I).$$
\end{proof}

We will use the following version of the Weierstrass Preparation Theorem in \cite{HKR00}.

\begin{lemma}\cite{HKR00} Let $A$ be a graded commutative ring, complete in the topology defined by powers of an ideal $I$. Suppose $\alpha(x) \in A[[x]]$ satisfies $\alpha(x) \equiv \omega x^d$ mod $(x^{d+1}, I)$, with $\omega \in A$ a unit. Then the ring $A[[x]]/(\alpha(x))$ is a free $A$-module with basis $\{1, x, x^2, \dots, x^{d-1}\}$.
\end{lemma}

In what follows, we apologize for the poor notation: the hat of completion and the hat in $\phat_1$ are unrelated. Recall the short exact sequence of Theorem \ref{thm:SES}:
$$0 \longrightarrow \imnr \longrightarrow ER(n)^*(\cp) \longrightarrow \frac{ER(n)^*[[\phat_1]]}{(\xi(\phat_1))}\longrightarrow 0$$

After completing at $I$, if we set $A={ER(n)^*}^{\wedge}_I$ we find that the Weierstrass Preparation Theorem applies to the right hand term.

\begin{prop} \label{prop:completion} ${ER(n)^*}^\wedge_I[[\phat_1]]/(\xi(\phat_1))$ is a free module over  ${ER(n)^*}^\wedge_I$ with basis $\{1, \phat_1, \phat_1^2, \dots, \phat_1^{2^{n-1}-1}\}$.

\end{prop}


\section{Appendix} \label{appendix}

The purpose of this section is to establish some technical lemmas that enable us to solve extension problems in the Bockstein spectral sequence. We construct a norm map
$$N_*:E(n)^*(X) \longrightarrow ER(n)^*(X)$$
and prove that (i) it is a map of $ER(n)^*(X)$-modules and (ii) the image of this map is represented on the $E_1$-page of the Bockstein spectral sequence by the group-theoretic norm. 

For more details and more generality, see \cite{LMS86}, \cite{RV14} (for equivariant orthogonal spectra), and the author's thesis \cite{Lor16}.

We begin with some general setup. Let $G$ be a finite group. Let $p: G \longrightarrow G/G=\{e\}$ denote projection onto the quotient. Let $\iota: \{e\} \longrightarrow G$ denote inclusion of the trivial subgroup. Let $U$ be a complete $G$-universe. Let $i: U^G \longrightarrow U$ denote the inclusion of universes. Let $GSU$ denote the category of (genuine) $G$-equivariant spectra indexed over $U$. Let $E \in GSU$ be a ring spectrum. Let $D=i^*E \in GSU^G$. 

We construct the norm $N:\iota^*D \longrightarrow D^G$ as follows. First, we construct $\Nt: EG_+ \wedge D \longrightarrow p^*(D^G)$. We get $N$ by applying $\iota^*(-)$ and noting that the map $EG_+ \wedge D \longrightarrow D$ is an equivalence on underlying nonequivariant spectra.

To construct ${\Nt:EG_+ \wedge D \longrightarrow p^*(D^G)}$, we start with the unit of the $(-/G, p^*)$ adjunction, $EG_+ \wedge D \longrightarrow p^*((EG_+ \wedge D)/G)$. As $EG_+ \wedge D$ is $G$-free, we may then compose with the Adams isomorphism (see \cite{LMS86}), 
$$A:p^*((EG_+ \wedge D)/G) \longrightarrow p^*((EG_+ \wedge D)^G).$$
Finally, we compose with ${p^*((EG_+ \wedge D)^G) \longrightarrow p^*(D^G)}$. We call this composite $\Nt$. 

$EG_+ \wedge D$ is a module over $p^*(D^G)$ with action given by
$$\xymatrix{EG_+ \wedge D \wedge p^*(D^G) \ar@{->}[r]& EG_+ \wedge D \wedge D \ar@{->}[r]& EG_+ \wedge D}.$$
$p^*(D^G)$ is a module over $p^*(D^G)$ with action given by
$$\xymatrix{p^*(D^G) \wedge p^*(D^G) \ar@{->}[r]& p^*((D\wedge D)^G) \ar@{->}[r]& p^*(D^G)}.$$
Thus, both the source and target of $\Nt$ are modules over $p^*(D^G)$ and a diagram chase shows that $\Nt$ is a map of $p^*(D^G)$-modules. Thus, $N: \iota^*D \longrightarrow D^G$ is a map of $D^G$-modules.

We now postcompose $N$ with the inclusion of fixed points $D^G \longrightarrow \iota^*D$. We thus have a self-map of $\iota^*D$ which we denote by $\N$. $\N$ has a very nice expression on the underlying nonequivariant homotopy groups of $D$, $\pi_*(\iota^*D)=\pi^u_*(D)$. It is given by composing the diagonal
$$\xymatrix{\iota^*D \ar@{->}[r]& G_+ \wedge \iota^*D \ar@{=}[r]& \bigvee_{g \in G}\iota^*D}$$
with the action map $G_+\wedge \iota^*D \longrightarrow \iota^*D$. For $z\in \pi_*^u(D)$ this gives
$$\N_*(z)=\sum_{g \in G}g \cdot z.$$

We now specialize to the case of interest. Let $G=C_2$. Let $X$ be a space with trivial $C_2$-action (in our present applications, $X=\cp$ with $C_2$ acting trivially and not by conjugation). Let $E=F(X, \ER(n))$. Then $E$ is a ring spectrum via the multiplication on $\ER(n)$. Note that the underlying nonequivariant spectrum is $\iota^*E=F(X, E(n))$. This is the source of the norm, with
$$\pi_*^uE=E(n)^*(X).$$
Since $C_2$ acts trivially on $X$, the target is
$$E^{C_2}=F(X, \ER(n))^{C_2}=F(X, \ER(n)^{C_2})=F(X, ER(n))$$
with
$$\pi_*E^{C_2}=ER(n)^*(X).$$
Thus, we have shown the following.
 
\begin{prop} \label{thm:norm} \text{ }
\begin{enumerate}[(a)] 
\item There is a norm map
$$N: F(X, E(n))=\iota^*E \longrightarrow E^{C_2}=F(X, ER(n)).$$
\item Give $F(X, ER(n))$ and $F(X, E(n))$ module structures over $F(X, ER(n))$ via the multiplication on $ER(n)$ and the map of ring spectra $ER(n) \longrightarrow E(n)$, respectively. Then $N$ is a map of modules over $E^{C_2}=F(X, ER(n))$. On homotopy, it is given as a map of modules over $ER(n)^*(X)$ by
$$N_*: E(n)^*(X) \longrightarrow ER(n)^*(X).$$
\item After composing with the inclusion of fixed points $ER(n)^*(X) \longrightarrow E(n)^*(X)$ (which is one of the maps in our exact couple) we denote the composite by $\N_*$,
$$\N_*: E(n)^*(X) \longrightarrow ER(n)^*(X) \longrightarrow E(n)^*(X).$$
It is given in homotopy by $\N_*(z)=z+c(z)$.
\item In the Bockstein spectral sequence, every class $\N_*(z)$ in $E(n)^*(\cp)=E_1^{0, *}$ is a permanent cycle represented by $N_*(z) \in ER(n)^*(X)$.
\end{enumerate}
\end{prop}

\bibliography{ernbib}{}
\bibliographystyle{alpha}

\end{document}